\documentclass[11pt,letterpaper]{amsart}
\usepackage{amsmath, amssymb, tikz, tikz-cd} 
\usepackage[colorlinks=true, linkcolor=red!80!black, urlcolor=purple, citecolor=blue!70!black]{hyperref}
\usepackage{mathrsfs}
\usepackage{graphicx,color}
\usepackage[all,cmtip]{xy}
\usepackage{enumitem}
\usepackage{verbatim}
\usepackage{bbm}
\usepackage{bm}
\usepackage{mathtools}
\usepackage{tabu}
\usepackage{comment}

\usepackage[mathscr]{eucal}

\newtheorem{theorem}{Theorem}[section]

\newtheorem{prop}[theorem]{Proposition}
\newtheorem{thm}[theorem]{Theorem}
\newtheorem{lem}[theorem]{Lemma}
\newtheorem{cor}[theorem]{Corollary}
\newtheorem{conj}[theorem]{Conjecture}
\newtheorem{thm-defn}[theorem]{Theorem-Definition}

\theoremstyle{definition}

\newtheorem{rem}[theorem]{Remark}

\newtheorem{expl}[theorem]{Example}
\newtheorem{defn}[theorem]{Definition}
\newtheorem{defn-prop}[theorem]{Definition-Proposition}
\newtheorem*{claim}{Claim}

\newcommand{\bP}{\mathbb{P}}
\newcommand{\bC}{\mathbb{C}}

\newcommand{\bZ}{\mathbb{Z}}

\newcommand{\Aut}{\mathrm{Aut}}
\newcommand{\vol}{\mathrm{vol}}
\newcommand{\Vol}{\mathrm{Vol}}

\newcommand{\cO}{\mathcal{O}}

\newcommand{\bw}{\mathbf{w}}

\newcommand{\ol}{\overline}

\newcommand{\bR}{\mathbb{R}}

\usepackage{graphicx}
\newcommand{\hvol}{\widehat{\mathrm{vol}}}

\newcommand{\cM}{\mathcal M}
\newcommand{\fm}{\mathfrak{m}}

\newcommand{\dps}{\displaystyle}

\newcommand{\wt}{\mathrm{wt}}

\newcommand{\Ric}{\mathrm{Ric}}

\newcommand{\bfa}{\mathbf{a}}
\newcommand{\bfx}{\mathbf{x}}

\title{Infinitely many families of Sasaki-Einstein metrics on spheres}

\author{Yuchen Liu}
\address{Department of Mathematics, Northwestern University, Evanston, IL 60208, USA}
\email{yuchenl@northwestern.edu}

\author{Taro Sano}
\address{Department of Mathematics, Graduate School of Science, Kobe university, 1-1, Rokkodai, Nada-ku, Kobe 657-8501, Japan} 
\email{tarosano@math.kobe-u.ac.jp} 

\author{Luca Tasin}
\address{Dipartimento di Matematica F.\ Enriques, Universit\`a degli Studi di Milano, Via Cesare Saldini 50, 20133 Milano, Italy} 
\email{luca.tasin@unimi.it}

\begin{document}

\maketitle

\begin{abstract}
We show that there exist infinitely many families of Sasaki-Einstein metrics on every odd-dimensional standard sphere of dimension at least $5$. We also show that the same result is true for all odd-dimensional exotic spheres that bound parallelizable manifolds.
\end{abstract}

\section{Introduction}

An odd-dimensional compact Riemannian manifold $(M,g)$ is called Sasakian if its (punctured) metric cone $(M \times \bR^{>0}, \bar{g})$ is K\"ahler, where $\bar{g}=r^2 g + dr^2$ is the conical metric on $M \times \bR^{>0}$ with $r$ the coordinate of $\bR^{>0}$. Denote by $X:=M\times \bR^{\geq 0}/M\times\{0\}$ the cone over $M$. Moreover, a Sasakian manifold $(M,g)$ is called Sasaki-Einstein if $g$ has constant Ricci curvature, or equivalently, if $\bar{g}$ is a Ricci-flat K\"ahler cone metric on $X$. The Reeb vector field $\xi$ on a Sasakian manifold $(M,g)$ is given by $\xi:=J(r\partial_r)$ where $J$ is the integrable complex structure on $X$. 

It is an important problem in differential geometry to find Sasaki-Einstein metrics on odd dimensional manifolds. As a starting  case,
the case of spheres and exotic spheres attracted much attention. In \cite{BGK05}, Boyer, Galicki, and Koll\'ar constructed a large (finite) number of inequivalent families of Sasaki-Einstein metrics on all homotopy spheres in dimension $4n+1$, or in dimension $7$. Using computer programming, \cite{BGKT05} generalizes this result to dimensions $11$ and $15$ for homotopy spheres that bound parallelizable manifolds. In \cite{CS19}, Collins and Sz\'ekelyhidi found infinitely many families of inequivalent Sasaki-Einstein metrics on $S^5$. In \cite{PJ}, Park and Won completely determined which simply connected rational homology 5-spheres admit Sasaki-Einstein metrics.

In this paper, we show the existence of infinitely many families of Sasaki-Einstein metrics on a large class of homotopy spheres in all (odd) dimensions. Our main result is the following. 

\begin{thm}\label{thm:main}
Let $n\geq 3$ be an integer. Then any $(2n-1)$-dimensional homotopy sphere which bounds parallelizable manifolds admits infinitely many inequivalent families of Sasaki-Einstein metrics.
\end{thm}

As a result, Theorem \ref{thm:main} implies the following two conjectures of Collins-Sz\'ekelyhidi and Boyer-Galicki-Koll\'ar.

\begin{cor}[{\cite[Conjecture 9.2]{CS19}}]\label{conj:CS}
There are infinitely many families of Sasaki-Einstein metrics
on every odd-dimensional standard sphere of dimension at least $5$.
\end{cor}

\begin{cor}[{\cite[Conjecture 4]{BGK05}}]\label{conj:BGK}
All odd-dimensional homotopy spheres that bound parallelizable manifolds admit Sasaki-Einstein metrics.
\end{cor}




\medskip

The strategy to prove Theorem \ref{thm:main} is based on a method developed in \cite{BG01,BGK05, BGKT05, CS19}, that we now outline (see the book \cite{book} for a detailed exposition on Sasakian geometry). 

Consider the link $L(\bfa):=Y(\bfa)\cap S^{2n+1}(1)$ of a Brieskorn-Pham singularity 
\[
Y(\bfa):= (z_0^{a_0} + z_1^{a_1}+\cdots + z_n^{a_n}=0)\subset \bC^{n+1}
\]
where  $n\geq 3$ and $\bfa = (a_0, \dots, a_n)\in \bZ_{>1}^{n+1}$. Then $L(\bfa)$ is a smooth $(2n-1)$-manifold which bounds a parallelizable manifold.
Under certain numerical conditions on $\bfa$, the link $L(\bfa)$ is a homotopy $(2n-1)$-sphere and so it is an element of the group $bP_{2n}$ 
of equivalence classes of oriented homotopy $(2n-1)$-spheres that bound parallelizable manifolds. See Section \ref{s:BP}.

Let $d:=\mathrm{lcm}(a_0,\cdots, a_n)$ and $d_i:= \frac{d}{a_i}$ for 
$i=0,\ldots ,n$.  
Let $\xi:= \sum_{i=0}^n d_i z_i \partial_{z_i}$ be a Reeb vector field on $Y(\bfa)$ of weight $(d_0, \ldots, d_n)$ which generates the corresponding $\bC^*$-action on $Y(\bfa)$.  
The following result 
shows that the Lichnerowicz obstruction \cite[(3.23)]{MR2318866} is not only necessary but also sufficient for the existence of Sasaki-Einstein metrics on $L(\bfa)$.
\begin{thm}\label{thm:k-stab}
Notation as above. Assume that $a_0\leq a_1\leq  \cdots \leq a_n$.
Then  $(Y(\bfa),\xi)$ is a K-polystable (resp. K-semistable) Fano cone singularity if and only if 
\begin{equation}\label{eq:maininequality}
 1<\sum_{i=0}^n \frac{1}{a_i}< ~(\textrm{resp. }\leq) ~1+ \frac{n}{a_n}.
\end{equation}
Moreover, the link $L(\bfa)$ admits a Sasaki-Einstein metric if the above strict inequalities hold.
\end{thm}
We note that \cite[Theorem 1]{MR2341843} studied the case where $a_0, \ldots ,a_n$ are pairwise coprime 
and \cite[Proposition 6.1]{ST21} studied the case where the corresponding weighted hypersurface is well-formed. 


The diffeomorphism type of $L(\bfa)$ is determined by the Arf invariant if $n$ is odd and by the signature $\tau (\bfa)$ of the Milnor fiber if $n$ is even.
In Section \ref{s:infinite}, we show that for appropriate choices of $\bfa$ we get each possible value in $bP_{2n}$ infinitely many times, proving in this way Theorem \ref{thm:main}. The hard case is when $n$ is even: the problem of computing $\tau (\bfa)$ is reduced to showing that a certain function counting integral points of a polytope is a quasi-polynomial.  We then use this to generalize a classical example of Brieskorn. 

In Section \ref{s:moduli} we look at the moduli space of Sasaki-Einstein metrics proving the following.

\begin{cor}\label{cor:unbdddim}
	For $m \ge 3$, let $ \Sigma^{4m-1} \in bP_{4m}$ be a homotopy sphere. 
	Then the dimensions of the families of Sasaki-Einstein metrics on $\Sigma^{4m-1}$ and $S^{4m-3}$ are both unbounded.
\end{cor}

This gives a negative answer to a question in \cite[Section 6]{MR2237105}. 

In Subsection \ref{ss:Euler}, we show that the contact structures induced by our construction on each homotopy sphere $\Sigma^{4m-1}$ ($m \ge 3$) belong to infinitely many families.

\subsection*{Acknowledgements}
We would like to thank Charles Boyer, J\'anos Koll\'ar, Yoshihiko Matsumoto, Yuji Odaka, Andrea Petracci, Chenyang Xu, Zhouli Xu, and Ruixiang Zhang for helpful discussions and comments. 
We would like to thank the referees for helpful comments. 

The first named author is partially supported by NSF Grant DMS-2148266 (formerly DMS-2001317). 
The second named author is partially supported by JSPS KAKENHI Grant Numbers JP17H06127, JP19K14509. 
The third named author is a member of the GNSAGA group of INdAM and is supported by the project PRIN 2020KKWT53.





\section{Brieskorn-Pham singularities}\label{s:BP}


\subsection{Topological background on Brieskorn manifolds}\label{ss:topback}
We fix some notation. Let $n\geq 3$ be an integer and let $\bfa = (a_0, \dots, a_n)\in \bZ_{>0}^{n+1}$ such that each $a_i>1$. Let 
\[
Y(\bfa):= (z_0^{a_0} + z_1^{a_1}+\cdots + z_n^{a_n}=0)\subset \bC^{n+1}
\]
and its link $L(\bfa):= Y(\bfa)\cap S^{2n+1}(1)$ which is called a {\it Brieskorn manifold}. 
We know that $L(\bfa)$ is a smooth $(2n-1)$-manifold \cite[Corollary 2.9]{MR0239612}. 
Let $M (\bfa)$ be the Milnor fiber of the polynomial $f(z) = z_0^{a_0} + z_1^{a_1}+\cdots + z_n^{a_n}$, 
that is, the fiber of the locally trivial fibration $\varphi \colon S^{2n+1}(1) \setminus L(\bfa) \to S^1(1)$ defined by 
$\varphi(z) = \frac{f(z)}{|f(z)|}$.  
Since the Milnor fiber is pararellizable (\cite[Theorem 5.1]{MR0239612}, \cite[(1.23) Proposition]{MR1194180}), we know that $L(\bfa)$ always bounds a parallelizable manifold. 

\begin{defn}
To each $\bfa$ as above, one associates a graph $G(\bfa)$ with $n+1$ vertices labeled by $a_0, \dots, a_n$. Two vertices $a_i$ and $a_j$ are connected in $G(\bfa)$ if and only if $\gcd(a_i, a_j)>1$. Let $G(\bfa)_{\rm ev}$ be the connected component of $G(\bfa)$ that contains all even integers. Note that $G(\bfa)_{\rm ev}$ may contain odd integers as well. 
\end{defn}

The following theorem gives a characterization of when $L(\bfa)$ is a homotopy $(2n-1)$-sphere, that is, a smooth $(2n-1)$-manifold homotopy equivalent to the $(2n-1)$-sphere. (By the topological generalized Poincar\'{e} conjecture, we can replace ``homotopy equivalent'' in the above with ``homeomorphic''.) 

\begin{thm} \cite[Satz 1 (ii)]{Bri66}\label{thm:brieskorn}
Let $n \ge 3$, $\bfa$, $L(\bfa)$ and $G(\bfa)$ be as above. The link $L(\bfa)$ is homeomorphic to the $(2n-1)$-sphere if and only if either of the following holds.
\begin{enumerate}
    \item $G(\bfa)$ contains at least two isolated points, or
    \item $G(\bfa)$ contains a unique odd isolated point, and $G(\bfa)_{\rm ev}$ has an odd number of vertices with $\gcd(a_i,a_j)=2$ for any distinct $a_i, a_j \in G(\bfa)_{\rm ev}$. 
\end{enumerate}
\end{thm}

\begin{rem}
When $L(\bfa)$ is homeomorphic to the $(2n-1)$-sphere, we call $L(\bfa)$ a {\it Brieskorn sphere}. 
When $n \ge 3$, every diffeomorphism type of a homotopy $(2n-1)$-sphere which bounds parallelizable manifolds 
can be realized by a Brieskorn sphere $L(\bfa)$ for some $\bfa$ (\cite{Bri66}, \cite[Theorem 9.4.8]{book}). 
\end{rem}

\subsection{Diffeomorphism types of Brieskorn spheres}
We continue to use the notations in \ref{ss:topback}. 

 Let $\Theta_{2n-1}$ be the Kervaire-Milnor groups consisting of equivalence classes of oriented homotopy $(2n-1)$-spheres  that are equivalent under oriented h-cobordism (see e.g. \cite{KM63}). We are interested in the subgroup $bP_{2n} < \Theta_{2n-1}$ consisting of equivalence classes of oriented homotopy $(2n-1)$-spheres that bound parallelizable manifolds. 

If $n = 2m~(m\geq 2)$ is even, we have the following theorem on the diffeomorphism type of $L(\bfa)$. 

\begin{thm}
Let $n = 2m~(m\geq 2)$ and $\bfa \in \bZ_{>0}^{n+1}$ satisfying the condition in Theorem \ref{thm:brieskorn} so that $L(\bfa) \in bP_{4m}$. 
\begin{enumerate}
\item[(i)] (\cite[p.526]{KM63}, \cite[Corollary 3.20]{10.1007/BFb0074439}) The group $bP_{4m}$ is cyclic of order \[
|bP_{4m}|= a_m 2^{2m-2}(2^{2m-1}-1)\cdot\mathrm{numerator}\left(\frac{4B_{m}}{m}\right), 
\] where $a_m = \begin{cases}
1 & (m \in 2 \bZ) \\
2 & (m \notin 2\bZ)
\end{cases}$ and $B_{m}$ is the $m$-th Bernoulli number. 
\item[(ii)](\cite[Theorem 7.5]{KM63}, \cite[Satz 3]{Bri66}) The class of $L(\bfa)$ in $bP_{4m}$ is $\frac{\tau (\bfa)}{8}$ with a suitable choice of the generator of $bP_{4m}$, where $\tau(\bfa)$ is is the signature of the intersection pairing on $H^n$ of the Milnor fiber $M(\bfa)$. Moreover, $\tau(\bfa)$ can be described as follows.   
\begin{align*}
\tau ( \bfa ) = &~  \# \left\{\bfx = (x_0, \dots , x_{2m}) \in \bZ^{2m+1}\mid 0<x_i<a_i\textrm{ and } 0<\sum_{i=0}^{2m} \frac{x_i}{a_i}<1 \mod 2 \right\}\\
& - \# \left\{\bfx =(x_0, \dots , x_{2m}) \in \bZ^{2m+1}\mid 0<x_i<a_i\textrm{ and } 1<\sum_{i=0}^{2m} \frac{x_i}{a_i}<2 \mod 2 \right\}. 
\end{align*}

\end{enumerate}
\end{thm}

If $n = 2m + 1$~($m \ge 1$) is odd and $\bfa$ satisfies the condition in Theorem \ref{thm:brieskorn}, 
then $L(\bfa)$ is a homotopy $(4m+1)$-sphere and belongs to $bP_{4m+2}$. The group $bP_{4m+2}$ has order $1$ or $2$, and it is trivial for $4m+1 \in \{ 1,5,13,29,61\}$, and order 2 for $4m+1\not\in \{1,5,13,29,61,125\}$ thanks to the almost completely solved Kervaire invariant problem \cite{Bro69, HHR16}. 

The Kervaire sphere is a homotopy $(4m+1)$-sphere whose Arf invariant equals $1$. If $bP_{4m+2}\cong \bZ_2$ then the Kervaire sphere is the only non-trivial element. Brieskorn also proved the following (see e.g. \cite[Theorem 9.4.3]{book}). 

\begin{thm} Let $n = 2m + 1$~($m \ge 1$) and $\bfa$ as above so that $L(\bfa) \in bP_{4m+2}$. 

 Then $L(\bfa)$ is the Kervaire sphere if and only if all of the following holds:  
\begin{itemize}
\item the condition (2) of Theorem \ref{thm:brieskorn} holds. 
\item the one isolated point, say $a_0$, satisfies $a_0\equiv \pm 3 \mod 8$. 
\item $G(\bfa) = G(\bfa)_{\rm ev}  \cup \{a_0 \}$. 
\end{itemize}
\end{thm}

We remark that our method can only produce Sasaki-Einstein metrics on homotopy spheres in $bP_{2n}$. The quotient group $\Theta_{2n-1}/bP_{2n}$ is the hard part of the group of homotopy spheres, which almost completely reduces to the computation of stable homotopy groups of spheres. For a recent account on this topic, see \cite[Section 1.4]{IWX20}.

\subsection{K-stability}
We refer the definition of Fano cone singularities and K-polystability/semistability to \cite[Section 2.2]{MR4301561}. 

Throughout the paper, we always work with quasi-regular Fano cone singularities $(Y,\xi)$, i.e. $Y$ is an affine variety with klt singularities with the vertex $y \in Y$, and  $\xi$ is a rational Reeb vector generating a good $\mathbb{C}^*$-action on $Y$ \cite[Definition 2.12]{MR4301561}. 

In this case, we have the quotient morphism $Y \setminus \{y \} \rightarrow X$ by the $\bC^*$-action and 
  the branch divisor $\Delta_X$ so that $(X, \Delta_X)$ is a log Fano pair \cite[Definition 2.13]{MR4301561}. 
 
 We shall use the following. 
 
 \begin{thm}\label{thm:equivKpolyst}
 Let $(Y, \xi)$ be a Fano cone singularity and $(X, \Delta_X)$ be its $\bC^*$-quotient as above. 
Then the following are equivalent. 

\begin{enumerate} 
\item[(i)] $(Y, \xi)$ is K-polystable (resp. K-semistable) . 
\item[(ii)] $(X, \Delta_X)$ is K-polystable (resp. K-semistable).  
\end{enumerate}
(See e.g. \cite[Definition 2.23, 2.28]{MR4301561} for the definitions of K-polystability/ K-semistability of Fano cone singularities and log Fano pairs. )

Assume that $Y \setminus \{y \}$ is smooth. 
 Then the K-polystability of $(Y, \xi)$ is equivalent to the following. 
\begin{enumerate}
\item[(iii)] $Y$ admits a Ricci-flat K\"{a}hler cone metric and the link $L_Y$ of $Y$ admits a Sasaki-Einstein metric with Reeb vector field $\xi$. 
\end{enumerate} 
 \end{thm}

\begin{proof}
\noindent(i) $\Leftrightarrow$ (ii): This  follows from the correspondence of special test configurations of $(Y, \xi)$ and $(X, \Delta_X)$ as in \cite[2.5]{Li:2021tr} and the equality of generalized Futaki invariants \cite[Lemma 2.30]{MR4301561} (cf. \cite[Theorem 2.9]{Li:2021tr}). 

Now assume that $Y \setminus \{y \}$ is smooth. 

 Then the last equivalence follows from \cite[Theorem 1.1]{CS19} and \cite[Corollary 11.1.8]{book}, for example (See also \cite{AJL21, Li:2021tr}). 
\end{proof}

We prove Theorem \ref{thm:k-stab} by modifying the argument in \cite[Proposition 6.1]{ST21}. 

\begin{proof}[Proof of Theorem \ref{thm:k-stab}] 
Let $(Y(\bfa), \xi)$ be the Fano cone singularity induced by the Reeb vector field $\xi$ on $Y(\bfa)$. 
Let $(X, \Delta_X)$ be the log Fano pair induced by the $\bC^*$-action on $Y(\bfa)$ as above.  
Then $(Y(\bfa), \xi)$ is K-polystable (K-semistable) if and only if $(X, \Delta_X)$ is K-polystable (K-semistable) by Theorem \ref{thm:equivKpolyst}. 
Hence we shall show that $(X, \Delta_X)$ is K-polystable (K-semistable) if and only if the inequality (\ref{eq:maininequality}) holds. 

We see that 
\[
\Delta_X= \sum_{j=0}^n \left(1- \frac{1}{g_j}\right) H_j 
\]
for $g_j:= \gcd (d_0, \ldots , \hat{d_j}, \ldots, d_n)$ and $H_j:= (z_j=0) \cap X$. Note that $X$ is isomorphic as a variety to the quasi-smooth weighted hypersurface 
\[
(y_0^{a_0'}+ \cdots + y_n^{a_n'} =0) \subset \bP \left(\frac{d_0 g_0}{g}, \ldots , \frac{d_n g_n}{g} \right)=: \bP',  
\] 
where $a'_i:= a_i/g_i$ and $g:= g_0 \cdots g_n$. 

    
Let $\Pi \colon \bP' \rightarrow \bP^n$ be the Galois cover determined by $$[y_0 : \cdots : y_n] \mapsto [y_0^{a_0'} : \cdots : y_n^{a_n'}]. $$ Let $L:= (w_0 + \cdots + w_n =0) \subset \bP^n$ and $\pi:= \Pi|_{X} \colon X \rightarrow H$ be the induced Galois cover. Then we have the ramification formula
\[
K_X = \pi^* \left(K_L+ \sum_{i=0}^n \left(1- \frac{1}{a'_i} \right)L_i \right)
\]
for $L_i:= L \cap (w_i=0)$ for $i=0, \ldots ,n$. By these and $a_j' H_j = \pi^*(L_j)$, we see that 
\begin{equation}\label{eq:ramifXL}
K_X+ \Delta_X = \pi^*\left(K_L + \sum_{i=0}^n \left(1- \frac{1}{a_i}\right)L_i\right). 
\end{equation}
By this, we see that the log Fano condition is equivalent to $1 < \sum_{i=0}^n \frac{1}{a_i}$. We also see that the K-polystability (resp. \ K-semistability) of $(X, \Delta_X)$ is equivalent to that of $(L, \sum_{i=0}^n (1- \frac{1}{a_i})L_i))$ by the above equality (\ref{eq:ramifXL}) and \cite[Theorem 1.2]{Liu:2020wo}, \cite[Corollary 4.13]{Zhuang}. 
Hence the statement is reduced to the following.

\begin{claim}
$ \left(L, \sum_{i=0}^n (1- \frac{1}{a_i})L_i \right)$ is K-polystable (resp.\ K-semistable) if and only if 
\[
1+ \frac{n}{a_n} -\sum_{i=0}^n \frac{1}{a_i}> ~(\textrm{resp. }\geq) ~0. 
\]
\end{claim}

\begin{proof}[Proof of Claim]
By \cite[Corollary 1.6]{Fujita:2017wm}, the K-polystability (resp.\ K-semistability) is equivalent to 
\begin{equation}\label{eq:lhs}
k \sum_{i=0}^n \left(1- \frac{1}{a_i}\right) - n \sum_{j=1}^k \left(1-\frac{1}{a_{i_j}}\right) > ~(\textrm{resp. }\geq) 0
\end{equation}
for all $1\le k \le n-1$ and $0 \le i_1 < \cdots < i_k \le n$. We obtain the claim since we have 
\begin{multline*}
\text{L.H.S. of (\ref{eq:lhs})} = k-k \sum_{i=0}^n \frac{1}{a_i} + n \sum_{j=1}^k \frac{1}{a_{i_j}} = k \left(1- \sum_{i=0}^n \frac{1}{a_i} + \frac{n}{k} \sum_{j=1}^k \frac{1}{a_{i_j}}\right) \\ 
\ge k \left(1- \sum_{i=0}^n \frac{1}{a_i}+ \frac{n}{a_n}\right). 
\end{multline*}
\end{proof}
Finally, the existence of Sasaki-Einstein metrics on $L(\bfa)$ follows from Theorem \ref{thm:equivKpolyst}. 
\end{proof}

\begin{rem}\label{rem:IYa}
Note that, for $Y(\bfa)$ as in Theorem \ref{thm:k-stab}, we have 
\[
I_{\bfa}:=I_{Y(\bfa)} := \sum_{i=0}^n \frac{d}{a_i} - d = d \left( \sum_{i=0}^n \frac{1}{a_i} -1 \right). 
\]
Hence the condition in Theorem \ref{thm:k-stab} is equivalent to 
$0 < I_{Y_{\bfa}} < n d_n$ as in \cite[Proposition 6.1]{ST21}. 
\end{rem}

We shall use the following propositions to distinguish the Sasaki-Einstein metrics on the links of weighted homogeneous hypersurface singularities.  

\begin{prop}\label{prop:linkandwhs} (\cite[Corollary 22]{BGK05}) 
For $i=1,2$, let 
\[
Y_i:= (F_i=0) \subset \bC^m
\]
 be an affine hypersurface defined by a weighted homogeneous polynomial 
$F_i$ with weights $\bw_i=(w_{i1}, \ldots , w_{im}) \in \bZ_{>0}^{m}$. 
Let $$X_i:= (Y_i \setminus \{0 \}) / \bC^*$$ be the quotient by the $\bC^*$-action on $Y_i$ of weights $\bw_i$. 
Assume that $X_i$ has a K\"{a}hler-Einstein orbifold metric and no orbifold holomorphic contact structure.  
Let $L_i$ be the corresponding link with the Einstein metric $g_i$. 

Then $(L_1, g_1)$ is isometric to $(L_2, g_2)$ if and only if the following conditions holds. 

\begin{itemize}
\item[(i)] $\bw_1 = \bw_2$ holds up to permutations. 
\item[(ii)] There exists an automorphism $\tau \in \Aut(\bC^m, \bw_1)$ as in \cite[15]{BGK05} such that $\tau(Y_1) = Y_2$ or its complex conjugate $\ol{Y}_2$. 
\end{itemize}
\end{prop}

\begin{prop}\label{prop:noofdcontact}
Let $0 \in Y=(F=0) \subset \bC^{n+1}$ be an isolated hypersurface singularity defined by a weighted homogeneous polynomial $F$ of degree $d$ for the $\bC^*$-action of weights $\bw = (w_0, \ldots ,w_n) \in \bZ_{>0}^{n+1}$. 
Let $X:= \left( Y \setminus \{0 \} \right) / \bC^*$ be the orbifold defined by the $\bC^*$-action. 

Then $X$ admits no orbifold holomorphic contact structure if one of the following holds: 
\begin{enumerate}
\item[(i)] $n$ is odd; 
\item[(ii)] $n=2m$ is even and $m \nmid  I_Y := \sum_{i=0}^n w_i -d$.  
\end{enumerate}
\end{prop}

\begin{proof}
This follows from \cite[\S 16]{BGK05}. 
\end{proof}

\subsection{Normalized volumes}
\begin{defn}[\cite{Li18}]
Let $y\in Y$ be a complex $n$-dimensional klt singularity. Let $v$ be a real valuation of $\bC(Y)$ centered at $y$, i.e. $v: \bC(Y)^\times \to \bR$ is a valuation satisfying $v(\bC^\times)=0$ and $v(\fm_y)>0$. We define the normalized volume of $v$ as 
\[
\hvol(v):=A_Y(v)^n \cdot \vol(v),
\]
where $A_Y(v)$ is the log discrepancy of $v$ according to \cite{JM12, BdFFU15}, and $\vol(v)$ is the volume of $v$ according to \cite{ELS03}. We define the local volume of the singularity $y\in Y$ as
\[
\hvol(y, Y):= \min_v \hvol(v),
\]
where $v$ runs over all real valuations of $\bC(Y)$ centered at $y$. The existence of a $\hvol$-minimizing valuation was proved by Blum \cite{Blu18}.
\end{defn}

Li's motivation of introducing the normalized volume was to study the local K-stability theory for klt singularities. We refer the readers to the surveys \cite{LLX20, Zhu23} for more backgrounds and discussions. 
Here we will focus on the case of K-polystable Fano cones. 

\begin{thm}\label{thm:nv-Kps}
Let $(Y,\xi)$ be a K-polystable Fano cone singularity with $y\in Y$ the cone vertex. Then $\hvol(y, Y) = \hvol(\wt_\xi)$, where $\wt_\xi$ is the valuation induced by the Reeb vector field $\xi$ (see e.g. \cite[(2.6)]{MR4301561}).

If moreover $Y\setminus\{y\}$ is smooth, then
\[
\hvol(y, Y) = n^n \cdot \frac{\Vol(L_Y, g_{L_Y})}{\Vol(S^{2n-1}, g_{\rm st})}.
\]
Here $g_{L_Y}$ is the Sasaki-Einstein metric on the link $L_Y$ of $Y$ whose existence follows from Theorem \ref{thm:equivKpolyst}, $g_{\rm st}$ is the standard metric on the round sphere $S^{2n-1}$, and $\Vol$ denotes the Riemannian volume.
\end{thm}

\begin{proof}
The first statement follows from \cite[Theorem 1.3]{LX18}. The second statement follows from the well-known equality $\hvol(\wt_\xi) = n^n \cdot \frac{\Vol(L_Y, g_{L_Y})}{\Vol(S^{2n-1}, g_{\rm st})}$, see e.g. \cite{MR2318866, MSY08, DS15, HS17, LX18}.
\end{proof}

\begin{expl}\label{eg:BPvolume}
Let $(Y(\bfa), \xi)$ be a K-polystable Brieskorn-Pham singularity with $a_0\leq \cdots\leq a_n$. Then by \cite[(3.18)]{MR2318866} and Theorem \ref{thm:nv-Kps} we have
\[
\hvol(0, Y(\bfa) ) = \frac{d (d_0+\cdots+d_n- d)^n}{d_0\cdots d_n}= \left(\sum_{i=0}^n \frac{1}{a_i} - 1 \right)^n\cdot \prod_{i=0}^n a_i .
\]
By Theorem \ref{thm:k-stab}, we have $0<\sum_{i=0}^n \frac{1}{a_i} - 1\leq \frac{n}{a_n}$ which implies $a_0\leq n$ and
\begin{equation}\label{eq:hvolineq}
\hvol(0, Y(\bfa) ) \leq n^n \prod_{i=0}^{n-1} \frac{a_i}{a_n}\leq n^n \frac{a_0}{a_n}\leq \frac{n^{n+1}}{a_n}.
\end{equation}

\end{expl}

\begin{rem}\label{rem:EinsteinFamily}
Let $\cM$ be the set of Riemannian metrics on a given compact smooth manifold $M$ of dimension $n$. 
For a Riemannian manifold $(M, g)$, we have the total scalar curvature functional on $\cM$
\[
S(g):= \int_M s_g \vol_g, 
\]
where $s_g$ is the scalar curvature and $\vol_g$ is the Riemannian volume form of $g$ \cite[4.16 Definition]{MR867684}. 
By \cite[4.21 Theorem]{MR867684}, we know that $g$ is Einstein if it is a critical point of the functional $\dps \frac{S(g)}{\Vol(g)^{\frac{n-2}{n}}}$, where $\Vol(g) := \int_M \vol_g$ is the Riemannian volume. 
Note that $s_g$ is (the Einstein) constant if $g$ is Einstein since $\Ric_g = s_g \cdot g$ in this case. 
By these, we see that $s_g \cdot \Vol(g)^{\frac{2}{n}}$ is constant on a continuous family of Einstein metrics. 
(See also \cite[12.52 Corollary]{MR867684}. )

Hence, in order to find an infinitely many families of Einstein metrics on a given manifold, 
it is enough to find a family of Einstein metrics with infinitely many values of $s_g \cdot \Vol(g)^{\frac{2}{n}}$. 
Note that, if $(M, g)$ is Sasaki-Einstein, then $s_g = \dim M -1$ (See e.g. \cite[Lemma 11.1.5]{book}). 
By these, Theorem \ref{thm:nv-Kps} and Example \ref{eg:BPvolume}, it is enough to exhibit K-polystable Fano cone singularities $(Y(\bfa), \xi)$ with infinitely many values of normalized volumes $\hvol(0, Y(\bfa) )$ to show Theorem \ref{thm:main} and Corollary \ref{conj:CS}. By the inequality (\ref{eq:hvolineq}), it suffices to find $\bfa$ with arbitrarily large $a_n$. 
\end{rem}

\section{Infinite sequence of K-polystable Fano cones}\label{s:infinite}

\subsection{Dimension $4m+1$}\label{ss:odd}
In this case, $n=2m+1$ with $m\geq 1$.  By the prime number theorem, for every sufficiently large prime number $p_n$, there are at least $n-2$ distinct prime numbers $p_2,\cdots, p_{n-1}$ in the interval $(\frac{(n-2)}{2(n-1)}p_n, \frac{1}{2}p_n)$. Let 
\[
a_0=a_1=2,~ a_i = 2p_i\textrm{ for }2\leq i\leq n-1, \textrm{ and } a_n= p_n.
\]
We can check that 
\[
1<\sum_{i=0}^n \frac{1}{a_i} = 1 + \sum_{i=2}^{n-1} \frac{1}{2p_i} + \frac{1}{p_n} < 1 + (n-2) \cdot \frac{(n-1)}{(n-2)p_n} + \frac{1}{p_n} = 1+\frac{n}{a_n}.
\]
Thus by Theorem \ref{thm:k-stab} the Fano cone $(Y(\bfa), \xi(\bfa))$ is K-polystable, i.e. $L(\bfa)$ admits a Sasaki-Einstein metric. Moreover, the condition of Theorem \ref{thm:brieskorn}(2) holds for the link $L(\bfa)$. Thus $L(\bfa)$ is a standard sphere if $p_n\equiv \pm 1\mod 8$, and is a Kervaire sphere if $p_n\equiv \pm 3 \mod 8$. By Dirichlet's theorem, there are infinitely many prime numbers $p_n$ satisfying each of the congruency conditions. Thus, by Remark \ref{rem:EinsteinFamily}, we find infinite families of Sasaki-Einstein metrics on both the standard sphere and the Kervaire sphere of dimension $4m+1$.
Note that since $n$ is odd,  $\left( Y(\bfa) \setminus \{0 \} \right)/\bC^*$ admits no holomorphic orbifold contact structure and so these metrics are pairwise non-isometric by Proposition \ref{prop:linkandwhs}. 

\subsection{Dimension $4m-1$} 
In this case, $n=2m$ with $m\geq 2$. 

\subsubsection{Standard spheres}\label{ss:standard}
 Let $p,q\geq 2$ be two positive integers. Let
\[
\bfa = (a_0, \dots, a_{n-2}, a_{n-1}, a_n) = (n-1, \dots , n-1, p, q). 
\]
By Theorem \ref{thm:brieskorn}, $L(\bfa)$ is a homotopy $(4m-1)$-sphere if $n-1$, $p$, and $q$ are pairwise coprime. Since $\sum_{i=0}^n \frac{1}{a_i}= 1+\frac{1}{p}+\frac{1}{q}$, by Theorem \ref{thm:k-stab} we know that $L(\bfa)$ admits a Sasaki-Einstein metric if $(n-1)p>q>n-1$ and $(n-1)q>p>n-1$. Note that, by  $d:=\mathrm{lcm} (a_0, \ldots ,a_n) = (n-1)pq$ and Remark \ref{rem:IYa}, we have 
\[
I_{\bfa} =(n-1)pq \left(\frac{1}{p} + \frac{1}{q}\right) = (n-1)(p+q). 
\]


For simplicity, denote by $s:= n-1 = 2m-1$.
First of all, we show that $\tau (\bfa)$ is an integer combination of $\gamma_j$ for $1\leq j\leq s$, where 
\[
\gamma_j : = \# \left\{(x,y)\in \bZ^2\mid 0<x<p, ~ 0<y<q, ~\textrm{ and } 0<\frac{x}{p}+\frac{y}{q} <\frac{j}{s} \right\}.
\]
Note that $\frac{x}{p} + \frac{y}{q}$ will never be of the form $\frac{j}{s}$ since $s, p , q$ are coprime. 
Then we may group $\bfx \in \bZ^{n+1}$ by the sum $\sigma = \sum_{i=0}^{n-2} x_i$, and for each value of $\sigma$ the corresponding terms in $\tau (\bfa)$ is an integer combination of $\gamma_j - \gamma_{j-1}$ for $1\leq j\leq 2s$. Furthermore, the involution $(x,y)\mapsto (p-x, q-y)$ yields $\gamma_j - \gamma_{j-1} = \gamma_{2s-j+1} - \gamma_{2s-j}$, so we only need $1\leq j\leq s$.

Next, we deduce a formula of $\gamma_j$ for special values of $p$ and $q$. Let $k$ be a positive integer, and let $p = sk+1$, $q = sp - 1 = s^2 k + s -1$. Clearly $s$, $p$, and $q$ are pairwise coprime satisfying $s<p<q<sp$, so $L(\bfa)$ is a Sasaki-Einstein homotopy sphere.
Then $\dps \frac{x}{p}+\frac{y}{q}<\frac{j}{s}$ implies that 
\[
x< \frac{jp}{s} = \frac{j(sk+1)}{s}   = jk +\frac{j}{s}.
\]
Hence $1\leq x\leq jk$. For each such $x$, we have 
\[
 y < q \left(\frac{j}{s}-\frac{x}{p}\right) = (sp - 1)\frac{pj-sx}{sp} = pj - sx - \frac{pj-sx}{sp}.
\]
Hence $1\leq y\leq pj-sx-1$ as $\frac{pj-sx}{sp}\in (0,1]$. Therefore,
\[
\gamma_j = \sum_{x=1}^{jk} (pj-sx-1) = \frac{jk}{2} (sjk +2j -s-2).
\]
Thus $\gamma_j$ is divisible by $\frac{k}{2}$ assuming $k$ is even. As a result, $\frac{\tau (\bfa)}{8}$ is divisible by $\frac{k}{16}$ assuming $16\mid k$. Thus for every $k$ that is a multiple of $16|bP_{4m}|$, we have that  $\frac{\tau (\bfa)}{8}$ is divisible by $|bP_{4m}|$, i.e. $L(\bfa)$ is diffeomorphic to the standard sphere. Thus we get an infinite sequence of $p,q$ where $L(\bfa)$ is a Sasaki-Einstein standard sphere.
Since 
\[
\frac{s+1}{2}=\frac{n}{2} \nmid I_{Y_{\bfa}}= (n-1)(p+q)=s(sk+k+1)=s(k(s+1)+1),
\]
$\left( Y(\bfa) \setminus \{0 \} \right)/\bC^*$ admits no holomorphic contact structure by Proposition \ref{prop:noofdcontact}.
This implies that the metrics we found are pairwise non-isometric by 
Proposition \ref{prop:linkandwhs} and gives infinitely many inequivalent families by Remark \ref{rem:EinsteinFamily}. 

\subsubsection{Exotic spheres}\label{sub:exotic}
For $k \in \bZ_{>0}$, let $\bfa_2:= (2, \ldots, 2, 3, 6k \pm 1) \in \bZ^{n+1}$.
 Brieskorn (\cite{Bri66}, \cite[Example 9.4.7]{book}) found that $L(\bfa_2)$ 
 satisfies  
\[
\frac{\tau(\bfa_2)}{8}=(-1)^mk. 
\] So if we run over all $k\in \bZ_{>0}$ we get all $(4m-1)$-spheres in $bP_{4m}$. 

Let us fix $k \in \{1, \ldots , |bP_{4m}| \}$. In the following, we study the sequence 
\[
a_0=a_1=2,~ a_2 = \cdots = a_{n-2}=p, ~ a_{n-1}=p+1, ~ a_n = p+l,
\]
where $p,l \ge 2$ are integers such that $p$ is even, and $p$, $p+1$, and $p+l$ are pairwise coprime.
In other words, $p$ is even, $\gcd(p,l)=1$ and $\gcd(p+1, l-1)=1$. 
Such $p, p+1, p+l$ can be realized by taking, for any $q \in \bZ_{\ge 0}$, \[
l:=6k-3 (\text{or } 6k-1), \ \  p:= ql(l-1) +2.
\]  
Let $\bfa_{p,l}:= (a_0, \ldots, a_n) = (2, 2, p, \ldots, p, p+1, p+l) \in \bZ^{n+1}$. 

For $p \gg 0$ (w.r.t.\ $n$ and $l$),  we have 
\[
1< \sum_{i=0}^n \frac{1}{a_i} = 1 + \frac{n-3}{p} + \frac{1}{p+1} + \frac{1}{p+l} < 1+ \frac{n}{p+l} = 1+ \frac{n}{a_n}, 
\]
thus see that $L(\bfa_{p,l})$ admits a Sasaki-Einstein metric by Theorem \ref{thm:k-stab}. 

Let us recall the definition of a quasi-polynomial. 

\begin{defn}
A function $f \colon \bZ_{\ge 0} \to \bZ$ is a quasi-polynomial if there exist $s \in \bZ_{>0}$ and polynomials $p_0, \ldots , p_{s-1}$ 
such that $f(n) = p_i(n)$ when $i \equiv n \mod s$. 
\end{defn}

 The main result of this section is the following. 
 
\begin{prop}\label{prop:tauqpoly}
Under the above assumptions, $\tau (\bfa_{p,l})$ is a quasi-polynomial in $p$ of period $l(l-1)$. 
\end{prop}

\begin{cor}\label{cor:polynomialinq}
$\tau(\bfa_{p,l}) \equiv (-1)^m k \mod |bP_{4m}|$ for infinitely many $p$. 

In particular, any homotopy sphere in $bP_{4m}$ admits infinitely many inequivalent families of Sasaki--Einstein metrics. 
\end{cor} 

\begin{proof}[Proof of Corollary \ref{cor:polynomialinq}]
Take  $k \in \{1,\ldots, |bP_{4m}| \}$ and $l=6k-3$ (or $l=6k-1$). Then for any $p \equiv 2 \mod ml(l-1)|bP_{4m}|$, that is, $p=ql(l-1) +2$, where $q=q'm|bP_{4m}|$ for some $q' \in \bZ_{>0}$, we get that $\tau(\bfa_{p,l}) \equiv (-1)^mk \mod |bP_{4m}|$  by Proposition \ref{prop:tauqpoly}. 

We shall confirm that the metrics on $L(\bfa_{p,l})$ are pairwise non-isometric. 
Note that $$d=\mathrm{lcm} (a_0, \ldots, a_n) = p (p+1)(p +l)$$ and 
\begin{align*}
I_{\bfa}&= d\left(\sum_{i=0}^n \frac{1}{a_i} -1\right) =(n-3)(p+1)(p+l) + p(p+l) + p(p+1) \\ &=n(p+1)(p+l) -3(p+1)(p+l)+ p(2p+l+1) 
\\ &\equiv -p^2 -2pl -2p -3l  \mod m.
\end{align*}

If $l=6k-3$, then $I_{\bfa} \equiv -42k +13 \mod m$, while if $l=6k-1$, then $I_{\bfa} \equiv -42k-1$. Hence, for a fixed $m$ and $k$, we can always choose $l=6k-3$ or $l=6k-1$ so that $m \nmid I_{\bfa}$. This implies that $X=\left( Y(\bfa) \setminus \{0 \} \right)/\bC^*$ admits no holomorphic contact structure by Proposition \ref{prop:noofdcontact} and so the metrics we found are pairwise non-isometric by Proposition \ref{prop:linkandwhs}. 
Remark \ref{rem:EinsteinFamily} implies the existence of infinitely many families. 

\end{proof}
\begin{proof}[Proof of Theorem \ref{thm:main}]
Just combine Subsection \ref{ss:odd} and Corollary \ref{cor:polynomialinq}.
\end{proof}

The rest of the section is devoted to the proof of Proposition \ref{prop:tauqpoly}. We will repeatedly use the following lemma.

\begin{lem}\label{lem:qpolysum}
Let $f(x,y )$ be a quasi-polynomial in $x$ and $y$ of period $l$ (that is, for $0 \le k,k' <l$, there exists a polynomial $f_{kk'}(x, y)$ such that $f(x,y) = f_{kk'}(x,y)$ when $x \equiv k, \ y \equiv k' \mod l$). 

Then $\sum_{x=0}^s f(x,y)$ is a quasi-polynomial in $s$ and $y$ of period $l$.
\end{lem}

\begin{proof}
In the proof, we may assume that $y \equiv k' \mod l$ for a fixed $k'$ so that, if $x \equiv k \mod l$, then 

\[
f(x,y) = \sum_{i=0}^m \sum_{j=0}^m a_{ij}^k x^i y^j 
\] 
for some $m$ and constants $a_{ij}^k$ for $i,j=0, \ldots, m$ and $k=0, \ldots, l-1$. 
Write
\begin{multline*}
\sum_{x=0}^s f(x,y)= \sum_{k=0}^{l-1} \sum_{\substack{0 \le x \equiv k \\ \mod l}}^{s} f(x,y) = \sum_{k=0}^{l-1} \sum_{\substack{0 \le x \equiv k \\ \mod l}}^{s} \left(\sum_{i=0}^m \sum_{j=0}^m  a_{ij}^k x^i y^j \right) \\
= \sum_{j=0}^m \sum_{i=0}^m  \sum_{k=0}^{l-1} a_{ij}^k y^j \sum_{\substack{0 \le x \equiv k \\ \mod l}}^{s} x^i. 
\end{multline*}	

It is hence enough to prove that
\[
 \sum_{\substack{0 \le x \equiv k \\ \mod l}}^{s} x^n
\]	
is a quasi-polynomial in $s$ of period $l$ for any non-negative integer $n$. 
We have 
\begin{align*}
\sum_{\substack{0 \le x \equiv k \\ \mod l}}^{s} x^n = \sum_{d=0}^{\left\lfloor \frac{s-k}{l} \right\rfloor}(dl + k)^n = \sum_{d=0}^{\left\lfloor \frac{s-k}{l} \right\rfloor} \sum_{i=0}^n \binom{n}{i} (dl)^{i} k^{n-i}=\sum_{i=0}^n \binom{n}{i} l^{i} k^{n-i}\sum_{d=0}^{\left\lfloor \frac{s-k}{l} \right\rfloor}d^i.
\end{align*}
For any $i=0, \ldots, n$ the sum $\sum_{d=0}^{\left\lfloor \frac{s-k}{l} \right\rfloor }d^i$ is a polynomial in $\left\lfloor \frac{s-k}{l} \right\rfloor$ of degree $i+1$ by Faulhaber's formula (See e.g. \cite[(2.2)]{MR1376174}) and $\left\lfloor \frac{s-k}{l} \right\rfloor$ is a quasi-polynomial in $s$ of period $l$.
Hence we obtain the required quasi-polynomialness. 
\end{proof}

For the proof of Proposition \ref{prop:tauqpoly} we first need two preliminary lemmas on counting integral points in polytopes.

\begin{lem}\label{lem:triangle}
Let $p$ and $l$ be positive integers and let $j$ be a non-negative integer.
Set $R:= \left\lfloor \frac{lj}{p} \right\rfloor$ and 
\[
\gamma_j:=\# \left\{(x,y)\in \bZ_{\geq 0}^2\mid \frac{x}{p+1}+\frac{y}{p+l}\leq \frac{j}{p} \right\}.
\]
 Then 
\begin{multline}\label{eq:gamma_j}
\gamma_j = \frac{1}{2}\left(j- \left\lfloor \frac{R}{l} \right\rfloor \right) \left(j+ \left\lfloor\frac{R}{l}\right\rfloor +1\right) + (j+1) (R+1) \\ + \sum_{r=0}^R \min\left\{ \left\lfloor \frac{R}{l} \right\rfloor, \left\lfloor\frac{1}{l-1} (j - (p+1) r + R) \right\rfloor \right\}.
\end{multline}
\end{lem}

\begin{proof}
As $\frac{x}{p+1}+\frac{y}{p+l}\leq \frac{j}{p}$, we have $x\leq\frac{j(p+1)}{p}= j + \frac{j}{p}$, i.e. $x\leq j  + \left\lfloor \frac{j}{p}\right\rfloor $. For each $x$, we have 
\[
y\leq  \left(\frac{j}{p} - \frac{x}{p+1}\right)(p+l) = j-x + \frac{lj}{p}-\frac{(l-1)x}{p+1}.
\]
In other words, $y\leq j-x + \left\lfloor \frac{lj}{p}-\frac{(l-1)x}{p+1} \right\rfloor$. Therefore,
\begin{align*}
\gamma_j & = \sum_{x=0} ^{j + \left\lfloor\frac{j}{p} \right\rfloor} (j-x) + 1+ \left\lfloor \frac{lj}{p}-\frac{(l-1)x}{p+1} \right\rfloor\\
& = \frac{1}{2}\left(j- \left\lfloor\frac{j}{p}\right\rfloor \right) \left(j+ \left\lfloor \frac{j}{p} \right\rfloor +1 \right) +\sum_{x=0}^{j + \left\lfloor \frac{j}{p} \right\rfloor} 1+ \left\lfloor \frac{lj}{p}-\frac{(l-1)x}{p+1} \right\rfloor.
\end{align*}

For any non-negative integer $r$ we have
\begin{align*}
&\left\lfloor \frac{lj}{p}-\frac{(l-1)x}{p+1} \right\rfloor\geq r \Longleftrightarrow \frac{lj}{p}-\frac{(l-1)x}{p+1} \geq r \Longleftrightarrow x\leq \frac{p+1}{l-1} \left(\frac{lj}{p}-r\right) \\
& \Longleftrightarrow x\leq j + \left\lfloor \frac{1}{l-1} \left(j - (p+1) r +\frac{lj}{p}\right) \right\rfloor
= j + \left\lfloor \frac{1}{l-1} (j - (p+1) r + R)\right\rfloor .
\end{align*}
Clearly $j + \left\lfloor \frac{1}{l-1} (j - (p+1) r + R)\right\rfloor \geq 0$ for every $0\leq r\leq R$. Also note that $\left\lfloor \frac{j}{p}\right\rfloor  = \left\lfloor \frac{R}{l}\right\rfloor $.
We have
\begin{align*}
    & \sum_{x=0}^{j + \left\lfloor \frac{j}{p}\right\rfloor } 1+\left\lfloor \frac{lj}{p}-\frac{(l-1)x}{p+1} \right\rfloor  
    = \sum_{r=0}^{R} \# \left\{ 0\leq x\leq j + \left\lfloor \frac{R}{l} \right\rfloor \mid \left\lfloor  \frac{lj}{p}-\frac{(l-1)x}{p+1} \right\rfloor \geq r \right\}\\
    & = \sum_{r=0}^{R} \# \left\{ 0\leq x\leq j + \left\lfloor \frac{R}{l} \right\rfloor \mid x\leq j + \left\lfloor \frac{1}{l-1} (j - (p+1) r + R)\right\rfloor \right\}\\
    & = \sum_{r=0}^R j+1 + \min\left\{\left\lfloor \frac{R}{l} \right\rfloor , \left\lfloor \frac{1}{l-1} (j - (p+1) r + R)\right\rfloor \right\}\\
    & = (j+1)(R+1) + \sum_{r=0}^R \min \left\{\left\lfloor \frac{R}{l} \right\rfloor , \left\lfloor \frac{1}{l-1} (j - (p+1) r + R)\right\rfloor \right\}.
\end{align*}
Putting all together, we get the desired formula.
\end{proof}

\begin{lem}\label{lem:lower}
Let $p$ be positive integer and let $l, \eta$ be a non-negative integer. Then 
\[
\delta_\eta:=\# \left\{\bfx\in \bZ^{n}\mid x_i\geq 0 \textrm{ and } \sum_{i=1}^{n-1} \frac{x_i}{p} + \frac{x_n}{p+l}\leq\eta \right\}
\]
is a quasi-polynomial in $p$ of period $l$ (meaning that $\delta_\eta$ is a polynomial in $p$ if $l=0,1$). 
\end{lem}

\begin{proof}
If $l=0$, then this follows from Ehrhart's theorem (see e.g. \cite[Theorem 3.23]{BR}), so we can assume $l \ge 1$.
Denote by $\sigma:=\sum_{i=1}^{n-1} x_i$. Then for any given integer $j \in [0, \eta p]$, the range of $y:=x_n$ such that $\sigma = \eta p - j$ is
\[
y \leq (p+l) \frac{j}{p} = j + \frac{lj}{p}, \quad \textrm{i.e. } y \leq j + R,
\]
where $R=\left\lfloor \frac{lj}{p}\right\rfloor $.
For each $\sigma$, the number of solutions of $(x_1, \cdots, x_{n-1})$ is $F(\sigma):= \binom{\sigma + n-2}{n-2}$. Thus
\begin{multline*}
\delta_\eta = \sum_{j=0}^{\eta p} F(\eta p - j)\cdot (j + R + 1) \\
= \sum_{j=0}^{\eta p} F(\eta p - j) \cdot (j+1) + \sum_{R=0}^{\eta l -1} R\cdot \sum_{j=e_R^-}^{e_R^+} F(\eta p - j) + F(0)\cdot \eta l, 
\end{multline*}
where $e_R^{\pm}$ are defined as 
\begin{equation*}
e_R^-:= \left\lceil \frac{pR}{l}\right\rceil, \quad e_R^+:=\left\lfloor \frac{p(R+1)-1}{l}\right\rfloor .
\end{equation*}

By Lemma \ref{lem:qpolysum}, the first term in $\delta_\eta$ is a polynomial in $p$ and in the second term, $\sum_{j=e_R^-}^{e_R^+} F(\eta p - j)$ is a polynomial in $\eta p - e_R^-$ and $\eta p - e_R^+$. The latter are quasi-polynomials in $p$ of period $l$. Thus $\delta_\eta$ is a quasi-polynomial in $p$ of period $l$.
\end{proof}

For $j, r, R \in \bZ_{\ge 0}$, let  \[
\gamma'_R(j):= \frac{1}{2}\left(j-\left\lfloor \frac{R}{l}\right\rfloor \right) \left(j+\left\lfloor \frac{R}{l}\right\rfloor +1\right) + (j+1) (R+1), 
\]
\[
\gamma''_{r,R}(j):= \min \left\{\left\lfloor \frac{R}{l} \right\rfloor , \left\lfloor \frac{1}{l-1} (j - (p+1) r + R)\right\rfloor \right\}
\]
so that $\gamma_j= \gamma'_R(j) + \sum_{r=0}^R \gamma''_{r,R}(j)$ for $R= \left\lfloor \frac{lj}{p}\right\rfloor $ as in Lemma \ref{lem:triangle}. 

\begin{proof}[Proof of Proposition \ref{prop:tauqpoly}]

The proof has several steps. Denote by $\sigma:=\sum_{i=2}^{n-2} x_i$, $x:=x_{n-1}$, and $y:=x_{n}$.  We start with the following reduction. 

\begin{claim}
It is enough to show that, for each integer $1\leq \eta \leq n-1$, the number
\[
\beta_\eta:=\# \left\{\bfx =(x_2,\ldots,x_n) \in \bZ^{n-1}\mid x_i\geq 0 \textrm{ and } \sum_{i=2}^{n} \frac{x_i}{a_i}\leq\eta \right\}
\]		
is a quasi-polynomial in $p$ with period $l(l-1)$.
\end{claim}

\begin{proof}[Proof of Claim]
Since  $a_0=a_1=2$, we have 
\begin{align*}
	\tau (\bfa_{p,l}) = &~  \# \left\{\bfx =(x_2,\ldots,x_n) \in \bZ^{n-1} \mid 0<x_i<a_i\textrm{ and } 1<\sum_{i=2}^{n} \frac{x_i}{a_i}<2 \mod 2 \right\}\\
	& - \# \left\{\bfx =(x_2,\ldots,x_n) \in \bZ^{n-1} \mid 0<x_i<a_i\textrm{ and } 0<\sum_{i=2}^{n} \frac{x_i}{a_i}<1 \mod 2 \right\}.
\end{align*}

For each $1\leq \eta\leq n-1$, let 
\[
\alpha_\eta:=\# \left\{\bfx =(x_2,\ldots,x_n) \in \bZ^{n-1}\mid 0<x_i<a_i\textrm{ and } 0<\sum_{i=2}^{n} \frac{x_i}{a_i}\leq \eta \right\}.
\]
Clearly, $$\alpha_\eta-\alpha_{\eta-1} = \# \{\bfx =(x_2,\ldots,x_n) \in \bZ^{n-1} \mid 0<x_i<a_i\textrm{ and } \eta-1<\sum_{i=2}^{n} \frac{x_i}{a_i}<\eta\},$$ as 
$$\sum_{i=2}^{n} \frac{x_i}{a_i}=\frac{\sigma}{p}+\frac{x}{p+1}+\frac{y}{p+l}$$ is not an integer by coprimeness for $p+1, p+l$ together with  $0<x<p+1$ or $0< y < p+l$. Thus $\tau (\bfa_{p,l}) = \sum_{\eta=1}^{n-1} (-1)^\eta (\alpha_\eta - \alpha_{\eta-1})$.  This reduces the problem  to show that $\alpha_\eta$ is a quasi-polynomial in $p$ of period $l(l-1)$ for each $1\leq \eta \leq n-1$. 

We define
\[
\beta_\eta':=\# \left\{\bfx =(x_2,\ldots,x_n) \in \bZ^{n-1} \mid x_i> 0 \textrm{ and } \sum_{i=2}^{n} \frac{x_i}{a_i}\leq\eta \right\}.
\]

If $U := \{\bfx\in \bZ^{n-1}\mid x_i> 0 \textrm{ and } \sum_{i=2}^{n} \frac{x_i}{a_i}\leq\eta\}$, and $U_i := \{\bfx\in U\mid x_i\geq a_i\}$ for $i=2,\ldots , n$, then 
$\alpha_\eta = \# (U\setminus \cup_{i=2}^{n} U_i)$. By inclusion-exclusion, we know that $\alpha_\eta$ is an integer combination of $\#U =\beta_\eta'$ and $\#\cap_{i\in I} U_i$ for subsets $I\subset \{2,3,\cdots,n\}$. By replacing $x_i$ with $x_i - a_i$, we see that $\#\cap_{i\in I} U_i = \beta_{\eta-\#I}$. Thus $\alpha_\eta$ is an integer combination of $\beta_{\bullet}$ and $\beta_{\bullet}'$, so it suffices to show quasi-polynomialness of $\beta$'s.
Furthermore, using inclusion-exclusion again we can express $\beta_\eta-\beta_\eta'$ as a integer combination of lower-dimensional counting numbers of the form $\beta_\eta$ (or $\delta_{\eta}$ as in Lemma \ref{lem:lower}). 
Hence, using induction on the dimension we reduce to only showing that $\beta_\eta$ are quasi-polynomials in $p$ and the claim is proven.
\end{proof}

We now focus on $\beta_\eta$. If we fix $\sigma\geq 0$, then the number of solutions $\sum_{i=2}^{n-2} x_i=\sigma$ with each $x_i\in \bZ_{\geq 0}$ equals $\binom{\sigma+n-4}{n-4}$. For each $0\leq j\leq \eta p$, let $\gamma_j$ be as in Lemma \ref{lem:triangle}. Denote by $F(x) := \binom{x+n-4}{n-4}$. 

Thus we have 
\[
\beta_\eta = \sum_{j = 0}^{\eta p} 
F(\eta p-j) \gamma_j.
\]

 For each $0\leq R\leq \eta l-1$, the range of $j$ is $e_R^-\leq j \leq e_R^+$ with
\begin{equation}\label{eq:eRpm}
e_R^-:= \left\lceil \frac{pR}{l}\right\rceil, \quad e_R^+:=\left\lfloor \frac{p(R+1)-1}{l}\right\rfloor .
\end{equation}
Clearly $e_R^\pm$ are linear quasi-polynomials in $p$ of period $l$ and $e_{R+1}^{-} = e_R^{+} +1$. 
Hence we have 
\[
\beta_\eta = \sum_{j=0}^{\eta p} F(\eta p - j)\gamma_j = F(0)\gamma_{\eta p} + \sum_{R=0}^{\eta l - 1} \sum_{j=e_R^-}^{e_R^+} F(\eta p - j) \gamma_j.
\]
By Pick's theorem (See e.g. \cite[Theorem 2.18]{BR}), $\gamma_{\eta p}$ is a polynomial in $p$. Thus the problem reduces to showing quasi-polynomialness of $\sum_{R=0}^{\eta l -1}\sum_{j=e_R^-}^{e_R^+} F(\eta p - j) \gamma_j$ which is equal to 
\[
 \sum_{R=0}^{\eta l -1}\sum_{j=e_R^-}^{e_R^+} F(\eta p - j) \gamma'_R(j)  
  + \sum_{R=0}^{\eta l -1}\sum_{j=e_R^-}^{e_R^+} \sum_{r=0}^R F(\eta p - j) \gamma''_{r,R}(j).  
\] The former sum is a quasi-polynomial in $p$ 
as $F(\eta p - j) \gamma'_R(j)$ is a polynomial in $j$ for a fixed $R$, thus, by Lemma \ref{lem:qpolysum}, $\sum_{j=e_R^-}^{e_R^+} F(\eta p - j) \gamma'_R(j)$ is a polynomial in $e_R^{\pm}$ which are quasi-polynomials in $p$. We only need to focus on the latter sum with $\gamma''_j$.
Suppose $r\in [0, \eta l-1]$ is fixed, and $R$ ranges in $[r, \eta l -1]$. Then $j$ ranges in $[e_r^-, \eta p -1]$. Note that  
\[
\sum_{R=0}^{\eta l -1}\sum_{j=e_R^-}^{e_R^+} \sum_{r=0}^R  F(\eta p - j) \cdot \gamma''_{r,R}(j)= \sum_{r=0}^{\eta l -1} \sum_{R=r}^{\eta l -1} \sum_{j=e_R^{-}}^{e_R^+} F(\eta p - j) \cdot \gamma''_{r,R}(j). 
\] Thus it reduces to showing quasi-polynomialness (in $p$ with period $l(l-1)$) of the following function
\begin{equation}\label{eq:min-1^}
\sum_{j=e_r^{-}}^{\eta p-1} F(\eta p-j)\cdot
\gamma''_{r,R}(j) =
 \sum_{R=r}^{\eta l -1} \sum_{j=e_R^{-}}^{e_R^+} F(\eta p-j)\cdot
\gamma''_{r,R}(j). 
\end{equation}

$\gamma_{r,R}''(j)$ can be described by the following. 

\begin{lem}\label{lem:min^}
\begin{enumerate}
    \item If $j\leq pr$, then $\left\lfloor \frac{R}{l} \right\rfloor \geq \left\lfloor \frac{1}{l-1} (j - (p+1) r + R)\right\rfloor $.
    \item If $j\geq  pr$, then $\left\lfloor \frac{R}{l} \right\rfloor \leq \left\lfloor \frac{1}{l-1} (j - (p+1) r + R)\right\rfloor $.
\end{enumerate}

\end{lem}

\begin{proof}
We look at the difference 
\[
l(l-1) \left(\frac{R}{l}-\frac{1}{l-1} (j - (p+1) r + R)\right) = l(p+1)r - lj - R.
\]
If $j\leq pr$, then $R= \left\lfloor \frac{lj}{p}\right\rfloor \leq lr$, so $l(p+1)r - lj - R\geq l(p+1)r - lpr - lr =0$.
If $j\geq pr$, then $R\geq lr$, so similarly $l(p+1)r - lj -R\leq 0$.
\end{proof}

Clearly $e_r^{-} = \left\lceil\frac{pr}{l}\right\rceil \leq pr$. 
From Lemma \ref{lem:min^} we know that \eqref{eq:min-1^} can be expressed as
\begin{align*}
\textrm{If }r\leq \eta -1 \colon & \sum_{j=e_r^-}^{pr-1} F(\eta p-j)\cdot
\left\lfloor \frac{R}{l} \right\rfloor + \sum_{j=pr}^{\eta p-1} F(\eta p-j)\cdot \left\lfloor \frac{1}{l-1} (j - (p+1) r + R)\right\rfloor .\\
\textrm{If }r\geq \eta\colon & \sum_{j=e_r^-}^{\eta p-1} F(\eta p-j)\cdot
\left\lfloor \frac{R}{l} \right\rfloor .
\end{align*}

Note that $pr-1= e_{lr-1}^+$, $pr=e_{lr}^-$, $\eta p -1=e_{\eta l-1}^+$, and that, if $r \le \eta-1$ (resp. $r \ge \eta$), then \eqref{eq:min-1^} becomes
\[
\sum_{R=r}^{lr-1} Q_1(p,R) + \sum_{R=lr}^{\eta l -1} Q_2(p,R) \ \ \left(\textrm{resp. } \sum_{R=r}^{\eta l -1} Q_1(p,R) \right),   
\]
where 
\[
Q_1(p, R):=\sum_{j=e_R^-}^{e_R^+} F(\eta p-j)\cdot
\left\lfloor \frac{R}{l} \right\rfloor \quad \textrm{and} \quad 
\]
\[
Q_2(p, R):=\sum_{j=e_R^-}^{e_R^+} F(\eta p-j)\cdot
\left\lfloor \frac{1}{l-1} (j - (p+1) r + R)\right\rfloor .
\]
 Hence it suffices to show quasi-polynomialness  of $Q_1(p,R)$ and $Q_2(p,R)$ (in $p$), where $r,R\in [0,\eta l -1]$ are both fixed in $Q_1$, and $r\in [0,\eta-1]$ and $R\in [lr,\eta l-1]$ are both fixed in $Q_2$. Hence we write $Q_1(p):= Q_1(p,R)$ and $Q_2(p):= Q_2(p, R)$. 

By Lemma \ref{lem:qpolysum}, $Q_1(p)$ is clearly a polynomial in terms of $p$ and $e_R^+, e_R^-$, both being a quasi-polynomial in $p$  of period $l$. Thus $Q_1(p)$ is a quasi-polynomial in $p$ of period $l$. 

Since $F(\eta p-j)\cdot
\left\lfloor \frac{1}{l-1} (j - (p+1) r + R) \right\rfloor$ is a quasi-polynomial in $j$ and $p$ of period $l-1$, by using Lemma \ref{lem:qpolysum} again, we see that $Q_2(p)$ is a quasi-polynomial in $e_R^{\pm}$ and $p$, thus $Q_2(p)$ is a quasi-polynomial in $p$ of period $l(l-1)$. 

The proof is finished.
\end{proof}

We end this section with a natural conjecture.

\begin{conj} \label{conj:qpoly}
Let $n$ be a positive integer. For each  $1\leq i\leq n$, let $b_i, c_i$ be non-negative integers  satisfying $(b_i, c_i)\neq (0,0)$. Define a function $P: \bZ_{\geq 0}\to \bZ$ as
\[
P(t):=\# \left\{\bfx\in \bZ^{n}\mid x_i\geq 0 \textrm{ and } \sum_{i=1}^{n} \frac{x_i}{b_i t + c_i} \leq 1\right\}.
\]
In other words, $P(t)$ is the number of lattice points in the closed simplex as the convex hull of  $\{\mathbf{0}\}\cup\{(b_i t + c_i) \mathbf{e}_i\}_{1\leq i\leq n}$ where $(\mathbf{e}_1,\cdots, \mathbf{e}_n)$ is the standard basis of $\bZ^n$. 
Then $P(t)$ is a quasi-polynomial for $t\in \bZ_{\geq 0}$.
\end{conj}

As we saw earlier, the proof of Proposition \ref{prop:tauqpoly} reduces to showing a special case of Conjecture \ref{conj:qpoly}. 

If all $(b_i, c_i)$ are proportional, then Conjecture \ref{conj:qpoly} follows from a classical theorem of Ehrhart (see e.g. \cite[Theorem 3.23]{BR}). For the general case, Conjecture \ref{conj:qpoly} is proved under the extra assumption of $t\gg 1$ in \cite{CLS10}.

\section{Dimension of the moduli spaces}\label{s:moduli}

By modifying examples in \ref{sub:exotic}, we obtain a family of Sasaki-Einstein metrics of arbitrarily large dimension on homotopy spheres. 

\begin{prop}\label{prop:perturbSE}
 Let $n \ge 5$ be an integer. 
 Let $p\gg 0$ and $l \in \bZ_{\ge 2}$ be as in \ref{sub:exotic} so that $p$ is even and $p, p+1, p+l$ are pairwise coprime.  
Let \[
Y:= (z_0^2 + z_1^2 + F_p + z_{n-1}^{p+1} + z_{n}^{p+l} =0 ) \subset \bC^{n+1}, 
\]
where $F_p \in \bC[z_2, \ldots, z_{n-2}]$ is a general homogeneous polynomial of degree $p$, that is, $Y$ is a deformation of $Y(\bfa_{p,l})$ for $$\bfa_{p,l}=(a_0, \ldots, a_n)=(2,2, p, \ldots ,p, p+1, p+l ).$$ 
Let $\xi$ be the Reeb field of weights $
(d_0, \ldots, d_n)$, where $d_j = d/ a_j$ for $$d:= p(p+1)(p+l)$$ and $j=0, \ldots ,n$.  

Then $(Y, \xi)$ is K-polystable. 
\end{prop}

\begin{rem}
Note that $Y$ gives a semi-universal deformations of $Y(\bfa_{p,l})$ with the $\bC^*$-action. 
\end{rem}

\begin{proof}
Let $(X, \Delta_X)$ be the log pair corresponding to the orbifold $Y/ \bC^*$ for the $\bC^*$-action induced by $\xi$. 
We check that 
$$
X \simeq (z_0^2 + z_1^2 + F_p + y_{n-1}+y_n=0) \subset 
\bP \left(\frac{p}{2}, \frac{p}{2}, 1, \ldots ,1, p,p\right)$$
 and $\Delta_X = (1- \frac{1}{p+1}) H_{n-1} + (1- \frac{1}{p+l}) H_n$ for $H_{n-1} := (y_{n-1} =0)$ and $H_n:= (y_n=0)$. 
As in Theorem \ref{thm:k-stab}, it is enough to show the K-polystability of $(X, \Delta_X)$. 

Let 
$$X_p:= (y_0+y_1 + F_p(y_2, \ldots, y_{n-2}) + y_{n-1}+y_n =0) \subset \bP(p,p,1,\ldots,1, p, p),$$
 where $y_0, \ldots , y_n$ are the coordinates of $\bP(p,p,1,\ldots,1, p, p)$. Let $\phi_p \colon X \rightarrow X_p$ be the Galois cover determined by $$
 \ [z_0: z_1: z_2: \cdots:z_{n-2}: y_{n-1}: y_{n}] \mapsto [z_0^2: z_1^2 : z_2 : \cdots : z_{n-2}: y_{n-1}: y_n].$$  
Let $$\Delta_p:= \frac{1}{2} \left( H'_0 + H'_1 \right)+ \left( 1- \frac{1}{p+1} \right) H'_{n-1} + \left(1- \frac{1}{p+l} \right) H'_n,$$
 where $H'_i := (y_i=0) \subset X_p$. Then we have 
\[
K_X+ \Delta_X = \phi_p^*(K_{X_p} + \Delta_p), 
\]
thus the K-polystability of $(X, \Delta_X)$ is reduced to that of $(X_p, \Delta_p)$. In fact, we shall show:  

\begin{claim}
\begin{enumerate}
\item[(i)] The automorphism group of $(X_p, \Delta_p)$ is finite. \item[(ii)]Moreover, the pair $(X_p, \Delta_p)$ is uniformly K-stable. 
\end{enumerate}
\end{claim}

\begin{proof}[Proof of Claim] 
For both statements, we may assume that $F_p = y_2^p + \cdots + y_{n-2}^p$ by the upper-semicontinuity and the openness of the uniform K-stability \cite{BL18, BLX19}. 
 
\noindent(i) This is a slight modification of the proof of \cite[Proposition 37]{BGK05}. 
We need to show that there are no continuous families of isomorphisms of the form 
\[
\tau_t(y_i) = y_i + \sum_{j \ge 1} t^j g_{ij}(y_0, \ldots , y_n)
\]
satisfying $\tau_t(y_0 y_1) = y_0 y_1$, $\tau_t(y_j) = y_j$ for $j=n-1, n$ and 
\[
 \tau_t(y_0+y_1) + \sum_{i=2}^{n-2} \tau_t(y_i)^p + \tau_t(y_{n-1}+y_n) = y_0+ y_1 + \sum_{i=2}^{n-2} y_i^p + y_{n-1} + y_n 
\]
since $\tau_t$ induces elements of $\Aut(X_p, \Delta_p)$.  
The last equality implies $$\tau_t(y_0 + y_1) = y_0 + y_1,$$ thus we check $\tau_t(y_i) = y_i$ for $i=0,1$ and $$\sum_{i=2}^{n-2} \tau_t(y_i)^p = \sum_{i=2}^{n-2} y_i^p.$$ 
By the same argument as in the proof of \cite[Proposition 37]{BGK05} and $p \ge 3$, we see that $\tau_t(y_i) = y_i$ for $i=2, \ldots, n-2$. This is a contradiction.

\noindent(ii) We use the same symbols $L, L_i, \ldots$ as in Theorem \ref{thm:k-stab}. Let $$L:= (w_0+ w_1+ \cdots + w_n =0) \subset \bP^n$$ with the Galois cover 
\[
\pi_p \colon X_p \rightarrow L ; \ [y_0: \cdots y_n] \mapsto [y_0:y_1:y_2^p: \cdots : y_{n-2}^p: y_{n-1}: y_n]
\]
and let 
\[
\Delta_L:= \sum_{i=0}^n \left(1- \frac{1}{a_i} \right) L_i. 
\]
Note that $(L, \Delta_L)$ is K-stable and $K_{X_p} + \Delta_p = \pi_p^*(K_L+ \Delta_L)$ as in the proof of Theorem \ref{thm:k-stab}.  By this and \cite[Theorem 1.2]{Liu:2020wo}, \cite[Corollary 4.13]{Zhuang}, we see that $(X_p, \Delta_p)$ is K-polystable. This implies that $(X_p, \Delta_p)$ is K-stable by (i) and hence uniformly K-stable by \cite{LXZ21}. 
%
\end{proof}
\end{proof}

\begin{proof}[Proof of Corollary \ref{cor:unbdddim}]
Let $(d_0, \ldots , d_n)$ and $d$ be as in Proposition \ref{prop:perturbSE} and 
let $\bP:=\bP(d_0, \ldots ,d_n)$. 
By Proposition \ref{prop:perturbSE} and \cite[(15.2)]{BGK05}, we see that the connected component containing the family of weighted hypersurfaces in $\bP$ induced by $Y$ as in Proposition \ref{prop:perturbSE} has dimension at least 
\begin{multline*}
h^0(\bP, \cO(d)) - \sum_{i=0}^n h^0(\bP, \cO(d_i)) \\ = \left(3+ 2 \binom{\frac{p}{2}+n-4}{n-4} + \binom{p+n-4}{n-4} +2 \right)\\
- \left(2\left(2+ \binom{\frac{p}{2}+n-4}{n-4}\right)+(n-3)^2 +2 \right) 
\\ = \binom{p+n-4}{n-4} -(n-3)^2 -1. 
\end{multline*}
This is unbounded for $p \gg 0$ and $n > 4$. 
This and Proposition \ref{prop:linkandwhs} imply that the dimension of a family of Sasaki-Einstein metrics on the link $L(\bfa_{p,l})$ induced by $Y$ is unbounded. 

When $n=2m \ge 6$ is even, for a given $\Sigma \in bP_{4m}$, 
there are infinitely many $p$ and $l$ such that $L(\bfa_{p,l})$ is diffeomorphic to $\Sigma$ by Corollary \ref{cor:polynomialinq}. 
When $n=2m-1 \ge 5$ is odd, the link $L(\bfa_{p,l})$ is a standard sphere $S^{4m-3}$ 
by the criterion in \ref{ss:topback}. By these, we finish the proof of Corollary \ref{cor:unbdddim}. 
\end{proof}

\begin{rem}
The construction in the proof of Corollary \ref{cor:unbdddim} does not provide families of unbounded dimensions on $S^5$, homotopy $7$-spheres, or the Kervaire spheres.
\end{rem}

\subsection{Mean Euler characteristic}\label{ss:Euler}
Consider the links $L(\bfa_{p,l})$ with the contact structure induced by the construction of Subsection \ref{sub:exotic}.  
We are going to compute the mean Euler characteristic $\chi_m$ of  $L(\bfa_{p,l})$.
As explained in  \cite[Section 4.5]{BMvK} and \cite[Section 5.8]{KvK}, this is an invariant of the contact structure. For the notation we    refer to \cite{BMvK,KvK}.

\begin{center}
	\begin{tabular}{ l| l| l| l}
		Link & Period & $\chi^{S^1}$	 & Frequency \\ \hline
		$L(2,2,p\ldots,p,p+1,p+l)$ & $p(p+1)(p+l)$ & $n$ & 1 \\ 
		$L(2,2,p+1,p+l)$ &  $2(p+1)(p+l)$ & 3 & $p/2-1$ \\
		$L(p+1,p+l)$ & $(p+1)(p+l)$ & 1 &  $p/2$ \\
		$L(2,2,p\ldots,p,p+l)$ &  $p(p+l)$ & $n-1$ & $p$ \\
		$L(2,2,p\ldots,p,p+1)$ & $p(p+1)$ & $n-1$ & $p+l-1$ \\
		$L(2,2,p+l)$ &  $2(p+l)$ & 2 & $p(p+1)/2 -3p/2$ \\
		$L(2,2,p+1)$ &  $2(p+1)$ & 2 & $p(p+l)/2 -3p/2 -l +1$\\
		$L(2,2,p,\ldots,p)$ &  $p$ & $\chi^{S^1}_p$ & $(p+1)(p+l) -2p-l$\\
		$L(2,2)$ &  2 & 2 & $\phi_2$		     
	\end{tabular}
\end{center}
where 
$$
\phi_2=\frac{p(p+1)(p+l)}{2} +2p-p^2-\frac{pl}{2}
$$

and $\chi^{S^1}_p$ is a polynomial in $p$ with leading term $p^{n-4}$ (to compute it one can use \cite[Theorem 3.10]{KvK}).

Note that all orbits except $L(2,2)$ and $L(2,2,p,\ldots,p)$ are rational homology spheres by Brieskorn graph theorem \cite[Theorem 39]{BK05}, so the computation of $\chi^{S^1}$ follows by \cite[Lemma 4.5]{FSvK}. 

The Maslov index of the principal orbit is (\cite[Formula 4.9]{BMvK})

\begin{multline*}
\mu_P = 2p(p+1)(p+l)\left(1+\frac{n-3}{p} +\frac{1}{p+1} + \frac{1}{p+l} -1\right) \\
=2\left( (n-3)(p+1)(p+l) + p(2p+l+1) \right).
\end{multline*}

Hence
\begin{align*}
\chi_m &= -\frac{2\phi_2 + \chi^{S^1}_p ((p+1)(p+l) -2p-l) }{\mu_P}  \\ 
& -\frac{ 2(p(p+l)/2 -3p/2 -l +1) + 2(p(p+1)/2 -3p/2)}{\mu_P} \\
	 &-\frac{(n-1)(p+l-1) +(n-1)p +p/2 + 3(p/2-1)+n}{\mu_P}.
\end{align*}

Since $\chi^{S^1}_p$ has leading term $p^{n-4}$, $\chi_m$ assumes infinitely many different values for $p\gg 0$. This implies that the corresponding contact structures are not isomorphic. 

\bibliographystyle{alpha}
\bibliography{ref}

\end{document}